\documentclass[a4paper,12pt]{article}
\usepackage{amsmath,amssymb,amsthm}
\usepackage{url}

\newtheorem{theorem}{Theorem}[section]

\newtheorem{proposition}[theorem]{Proposition}
\newtheorem{lemma}[theorem]{Lemma}

\newtheorem{corollary}[theorem]{Corollary}
\newtheorem{definition}[theorem]{Definition}

\numberwithin{equation}{section}
\numberwithin{theorem}{section}

\newcommand{\Vol}{\mathop{\mathrm{Vol}}}	
\newcommand{\diff}{\mathbin{\Delta}}		

\title{Existence and uniqueness of the stationary measure in the continuous 
Abelian sandpile}

\urldef\email\url|{wkager,hliu,rmeester}@few.vu.nl|

\author{Wouter Kager\footnote{VU University Amsterdam,
De Boelelaan 1081a, 1081\,HV Amsterdam, The Netherlands, e-mail: \email}
\and Haiyan Liu$^\ast$ \and Ronald Meester$^\ast$}

\begin{document}

\maketitle

\begin{abstract}
Let $\Lambda \subset \mathbb{Z}^d$ be finite. We study the following sandpile model on~$\Lambda$. The height
at any given vertex $x \in \Lambda$ is a positive real number, and additions are uniformly distributed
on some interval $[a,b] \subset [0,1]$. The threshold value is~1; when the height at a given vertex exceeds~1, it topples, that
is, its height is reduced by~1, and the heights of all its neighbours in 
$\Lambda$ increase by $1/2d$. We first
establish that the uniform measure~$\mu$ on the so called `allowed 
configurations' is invariant under the dynamics. When $a < b$, we show with 
coupling ideas that starting from any initial configuration of heights, the 
process converges
in distribution to~$\mu$, which therefore is the unique invariant measure for 
the process.  When
$a=b$, that is, when the addition amount is non-random, and $a \notin 
\mathbb{Q}$, it is still the case that $\mu$ is the unique invariant 
probability
measure, but in this case we use random ergodic theory to prove this; this proof proceeds in a very different way. Indeed, the coupling approach cannot work in this case since we also show the somewhat surprising fact that when $a=b\notin \mathbb{Q}$, the process does not converge in distribution at all
starting
from any initial configuration.
\end{abstract}

\section{Introduction}

The classical {\em BTW-Sandpile Model} was introduced by Bak, Tang and Wiesenfeld
in 1988 as a prototype model for what they called `self-organized
criticality'~\cite{bak}. It has attracted attention from both physicists and mathematicians. In Dhar's paper~\cite{dhar}, the model was
coined the {\em Abelian Sandpile Model} (ASM) because of the Abelian
structure of the toppling operators. The BTW-sandpile model
is defined on a finite subset~$\Lambda$ of the $d$-dimensional integer lattice. At
each site of~$\Lambda$, we have a non-negative integer value that
denotes the {\em height} or the {\em number of `grains'} at that site. At each
discrete time~$t$, one grain is added to a random site in~$\Lambda$;
the positions of the additions are independent of each other. When a site has at least $2d$ grains, it is defined {\em unstable} and {\em topples}, that is,
it loses $2d$ grains, and each of its nearest neighbours in~$\Lambda$ receives 
one grain. Neighbouring sites that have received a grain, can
now become unstable themselves, and also topple. This is continued until all heights are at most $2d-1$ again. This will certainly happen, since
grains at the boundary of the system are lost. See
\cite{dhar, meester, Redig} for many detailed results about this model.

As a variant of the BTW model, the {\em Zhang sandpile model} was
introduced by Zhang~\cite{zhang} in 1989.
The Zhang sandpile model has some of the flavour of the BTW sandpile, albeit with some important and crucial differences: 
(i) in the Zhang-model, the height variables are continuous with values in $[0, \infty)$ and the threshold
is (somewhat arbitrarily) set to~1 (we also speak of {\em mass} rather than of height sometimes); (ii) at each
discrete time, a random amount of mass, which is uniformly
distributed on an interval $[a, b]\subset[0, 1)$, is added
to a randomly chosen site; (iii) when a site has height larger than~1, it is defined unstable and topples. In this model
this means that it loses {\em all}
its mass and each of its nearest neighbours in~$\Lambda$ receives a
$1/2d$ proportion of this mass. 

Since the mass
redistributed during a toppling depends on the current
configuration, the topplings in Zhang's sandpile are not abelian. From simulations in~\cite{zhang}, it is shown
that in large volume under stationary measure, the height variables of this model concentrate on several discrete `quasi-values', and in this sense, this model behaves similarly to the BTW-sandpile model. Recently, this model as well
as the infinite volume version of this model have been rigorously
studied in~\cite{anne, anne2}.

In the present paper, we discuss a model introduced in~\cite{azimi}, which is continuous like Zhang's sandpile, but
abelian like the BTW sandpile. The model is related to the deterministic model introduced in~\cite{gab}. In our model, the threshold value of all sites is again set to~1. 
The only difference with Zhang's model, is that when a site topples, it does not lose all its mass, but rather only a total of mass~1 instead. Each neighbour
in~$\Lambda$ then receives mass $1/2d$. Note that this is more similar to the BTW sandpile than Zhang's sandpile: in the BTW model the number of grains involved in a toppling does not depend on the actual height of the vertex itself either. In this sense, the newly defined model is perhaps
a more natural analogue of the BTW model than Zhang's model. In this paper we therefore call the new model the CBTW sandpile, where the
`C' stands for `continuous'.

In this paper, we study existence and uniqueness of the invariant probability measure for the CBTW sandpile. 
In the next section we formally define the model, set our notation and state our results. After that we 
study the relation between the CBTW and the BTW-sandpile model, used mainly
as a tool in the subsequent proofs. 

\section{Definitions, notation and main results}

Let $\Lambda$ be a finite subset of $\mathbb{Z}^d$ and let $\eta=(\eta(x), x \in \Lambda)$ be a configuration of heights, taking values in $\mathcal{X}=[0, \infty)^{\Lambda}$. Site~$x$ is called {\em stable} (in~$\eta$) if
$0 \leq \eta(x)<1$; if $\eta(x) \geq 1$, site~$x$ is called {\em unstable}. A configuration~$\eta$ is called stable
if every site in~$\Lambda$ is stable.
We define $\Omega=[0, 1)^\Lambda$ to be the collection of stable configurations.

By $T_x$, we denote the {\em toppling operator} associated to site~$x$: if $\Delta$ is the $|\Lambda|\times|\Lambda|$ matrix 
$$
\Delta(x, y)= \left\{
\begin{array}{ll}
1 & \mbox{ if } x=y,\\
- \frac{1}{2d} & \mbox{ if } |y-x|=1,\\
0 & \mbox{ otherwise},
\end{array}
\right.
$$
then
$$
T_x\eta: =\eta-\Delta(x, \cdot).
$$
Therefore, $T_x\eta$ is the configuration obtained from~$\eta$ by performing one toppling at site~$x$. 
The operation of~$T_x$ on~$\eta$ is said to be {\em legal} if $\eta(x)\geq1$,
otherwise the operation is said to be {\em illegal}.
It is easy to see that toppling operators commute, that is, we have

\begin{equation}
T_xT_y\eta=T_yT_x\eta.
\end{equation}

This abelian property of the topplings implies a number of useful properties, the proofs of which are similar to the
proof of Theorem~2.1 in~\cite{meester} and are therefore not repeated here.

\begin{proposition}
Let $\eta\in\mathcal{X}$ and suppose we start the sandpile dynamics with initial configuration~$\eta$. Then we have that
\begin{enumerate}
\item the system evolves to a stable configuration after finitely
many legal topplings; 
\item the final configuration is the
same for all orders of legal topplings; 
\item for each $x\in\Lambda$, the number of legal topplings at site~$x$
is the same for all sequences of legal topplings that result in a
stable configuration.
\end{enumerate}
\label{proposition1}
\end{proposition}
 
It follows that  for $\eta\in\mathcal{X}$, there is a unique stable configuration $\eta'\in\Omega$ reachable from~$\eta$ by 
a series of legal topplings. We define the {\em stabilization operator}~$\mathcal{S}$ as a map from~$\mathcal{X}$ to~$\Omega$  such that
\begin{equation}
\mathcal{S}\eta: =\eta'.
\end{equation}

For $x\in\Lambda, u\in[0, 1)$, $A_x^u$ denotes the {\em addition operator} defined by
\begin{equation}
A_x^u\eta=\mathcal{S}(\eta+u\delta_x),
\end{equation}
where 
$$
\delta_x(y)=\left\{
\begin{array}{ll}
0 & \mbox{ if } y\neq x,\\
1& \mbox{ if } y=x,
\end{array}
\right.
$$
that is, we add mass~$u$ at site~$x$ and subsequentially stabilize.
According to Proposition~\ref{proposition1}, $A_x^u$ is well-defined, and satisfies
$$
A_x^u A_y^v\eta= A_y^v A_x^u\eta.
$$

The evolution of the CBTW model can now concisely be described via 
\begin{equation}
\eta_t=A_{X_t}^{U_t}\eta_{t-1}, \quad t=1,2,3,\cdots,
\label{dynamic}
\end{equation} 
where $\eta_0$ is the initial configuration, $X_1, X_2,\ldots$ is a sequence of i.i.d.\ uniformly distributed random variables on~$\Lambda$ and $U_1, U_2, \ldots$ are i.i.d.\ and uniformly distributed on $[a,b]$, where $[a,b] \subset [0,1]$. The two sequences are also independent of each other. 

For $\eta\in\Omega$ and $W \subseteq \Lambda$, $\eta|_W$ denotes the restriction of~$\eta$ to~$W$.
For every finite subset~$W$ of~$\Lambda$, $\eta|_W$ is called a {\em forbidden sub-configuration} (FSC) if for all $x\in W$
$$
\eta(x)< \frac{1}{2d} \#(\textrm{nearest neighbours of } x \textrm{ in } W),
$$
where $\#$ denotes cardinality.
A configuration $\eta\in\Omega$ is said to be {\em allowed} if it has no FSC's. This notion is parallel to the corresponding notion
in the BTW-sandpile model---see also below. It is well known that in the BTW model, the uniform measure on all allowed configurations is invariant under
the dynamics (and is the only such probability measure). 
Let us therefore denote by~$\mathcal{R}$ the set of all allowed configurations and by~$\mu$ normalized Lebesgue measure on~$\mathcal{R}$, that is, 
the measure which assigns mass such that for every measurable $B\subset\mathcal{R}$ we have
\begin{equation}
	\mu(B) =\frac{\Vol(B)}{\Vol(\mathcal{R})},
\end{equation}
where $\Vol(\cdot)$ denotes Lebesgue measure.

For every initial probability measure~$\nu$, $\nu_t^{a, b}$ denotes the distribution of the process  at time~$t$. The first result (similar
to the corresponding result in the deterministic model in~\cite{gab}) states 
not only that $\mu$ is invariant for the process, but moreover that $\mu$ is 
in fact invariant for each individual transformation~$A_x^u$:

\begin{theorem}
For every $x\in\Lambda, u\in[0, 1)$, $A_x^u$ is a bijection on~$\mathcal{R}$ 
and $\mu$ is $A_x^u$-invariant. Hence, $\mu$ is invariant for the 
CBTW-sandpile model.
\label{invariant}
\end{theorem}

When $a  <  b$, the situation is  as in the traditional BTW model:

\begin{theorem}
For every $0 \leq a< b< 1$, $\mu$ is the unique invariant probability measure of the CBTW model and starting from every measure~$\nu$ on~$\Omega$, $\nu_t^{a, b}$ converges exponentially fast in total variation to~$\mu$ as $t$ tends to infinity.
\label{a<b}
\end{theorem}

In the case $a=b$ however, things are more interesting:

\begin{theorem}
When $a = b \in [0,1)$ but $a \not\in \{\frac{l}{2d}: l=0, 1, \dotsc, 2d-1\}$, 
for every initial configuration $\eta\in\Omega$, the distribution of the 
process at time~$t$ does not converge weakly at all as $t \to \infty$.
\label{a=b}
\end{theorem}

\begin{theorem}
	When $a = b$ and $a \in \{\frac{l}{2d}: l=0,1,\dotsc,2d-1\}$, for every 
	initial configuration $\eta\in\Omega$, the distribution of the process at 
	time~$t$ converges exponentially fast in total variation to a 
	measure~$\mu^\eta_a$ as $t\to\infty$.
	\label{a=l/2d}
\end{theorem}

In the proof of Theorem~\ref{a=l/2d} below, we give an explicit description of 
the limiting distribution~$\mu^\eta_a$ in terms of the uniform measure on the 
allowed configurations of the BTW model. In general, the limiting distribution 
depends both on the initial configuration~$\eta$ and the value of~$a$. 
Finally we have

\begin{theorem}
When $a\in[0, 1)$ is irrational and $a=b$, $\mu$ is the unique invariant (and ergodic) probability measure for the CBTW model.
\label{irrationala}
\end{theorem}

Hence, it is always the case that $\mu$ is the unique invariant probability 
measure for the CBTW model, except possibly if $a=b \in \mathbb{Q}$. However, 
convergence to this unique stationary measure only takes place when $a < b$, 
and not (apart from the obvious exceptional case when we start with~$\mu$) in 
the case $a=b$. In the proof of Theorem~\ref{a<b}, we will use the fact that 
$a$ is strictly smaller than~$b$ in order to construct a coupling. The proof 
of Theorem~\ref{irrationala} is very different from the proof of 
Theorem~\ref{a<b} (although both results are about the uniqueness of the 
invariant measure). In the case $a=b$ we use ideas from random ergodic theory 
rather than a coupling. We could have done this also in the proof of 
Theorem~\ref{a<b} but then we would not have obtained the exponential 
convergence corollary. Theorem~\ref{irrationala} is by far the most difficult 
to prove here, having no coupling approach at our disposal. In the next 
section, we discuss the relation between the CBTW and the BTW models. This 
relation will be used in the subsequent proofs of our results.

\section{Relation between the CBTW model and the BTW model}

We denote by $\mathcal{X}^o$ and $\Omega^o$ the collection of all height configurations and stable configurations in the classical BTW model,
respectively. Furthermore, we use the notation $t_x$, $a_x$ and $\mathcal{S}^o$ to denote the toppling operator at~$x$, the
addition operator at~$x$ and the stabilization operator in the BTW model, respectively. 

In order to compare the CBTW and BTW models, it turns out to be very useful to define the {\em integer} and the {\em fractional} part
of a configuration $\eta \in \mathcal{X}$.

\begin{definition} For a configuration $\eta\in\mathcal{X}$, $\overline{\eta}$ is defined by
\begin{equation}
\overline{\eta}(x)=\lfloor 2d\eta(x)\rfloor,
\end{equation}
and we call $\overline{\eta}$ the {\em integer part} of~$\eta$. We define $\tilde{\eta}$ by
$$
\tilde{\eta}(x)= \eta(x) \bmod \frac1{2d},
$$
and we call $\tilde{\eta}$ the {\em fractional part} of~$\eta$.
\label{conjugate}
\end{definition}
The identity
\begin{equation}
\eta=\frac{1}{2d}\overline{\eta}+\tilde{\eta}
\end{equation}
clearly holds. The evolution of the integer part of~$\eta$ is closely related to a BTW sandpile, while the fractional part is invariant under topplings. These
two features will make these definitions very useful in the sequel. 

Clearly, $\eta(x)\geq 1$ if and only if $\overline{\eta}(x)\geq 2d$. The operation of taking integer parts commutes with legal topplings:

\begin{lemma} For $\eta\in\mathcal{X}$ such that $\eta(x)\geq 1$, we have
$$
\overline{T_x\eta} = t_x \overline{\eta}
$$ 
and hence
$$\overline{\mathcal{S}\eta}=\mathcal{S}^o\overline{\eta}.
$$
\label{relation}
\end{lemma}

\begin{proof}
For $\eta\in\mathcal{X}$, $\eta(x)\geq 1$ implies
$\overline{\eta}(x) \geq 2d$.  From the toppling rule, we have 
$$
{T_x\eta}={\eta-\delta_x+\sum_{y \in \Lambda: |y-x|=1}\frac{1}{2d}\delta_y},
$$
which implies that for $z\in\Lambda$,
$$
\overline{T_x\eta}(z)=\lfloor 2d\eta(z)-2d\delta_x(z)+\sum_{y \in \Lambda: |y-x|=1}\delta_y(z)\rfloor.
$$
Since both $2d\delta_x(z)$ and $\sum_{y \in \Lambda: |y-x|=1}\delta_y(z)$ are
integers, we get
$$
\overline{T_x\eta}(z) = \lfloor 2d\eta(z)\rfloor - 2d\delta_x(z)
+ \sum_{y \in \Lambda: |y-x|=1}\delta_y(z)
$$
which is equal to $(t_x\overline{\eta})(z)$. Therefore we obtain 
$\overline{T_x\eta}=t_x\overline{\eta}$.
It now also follows that
$\overline{\mathcal{S}\eta}=\mathcal{S}^o\overline{\eta}$.
\end{proof}

For a configuration $\xi\in\mathcal{X}^o$, let $C(\xi)$ be the set 
$$
C(\xi):=\{\eta\in\mathcal{X}: \overline{\eta}=\xi\},
$$
that is, $C(\xi)$ is the set of all configurations which have
$\xi$ as their integer part. The sets $C(\xi), \xi \in \Omega^o$, partition~$\Omega$. Indeed it is easy to see that
for $\xi,\xi'\in\Omega^o$ with $\xi\neq\xi'$, we have
\begin{equation}
C(\xi)\cap C(\xi')=\emptyset
\end{equation}
and
\begin{equation}
\Omega=\bigcup_{\xi\in\Omega^o}C(\xi).
\end{equation}

Before continuing to the next observation, we recall the definition of a BTW-forbidden sub-configuration.
For $\xi\in\Omega^o$, $\xi|_W$ is a FSC if for all $x\in
W$
$$
\xi(x)<\#(\textrm{ nearest neighbours of } x \textrm{ in } W).
$$
By $\mathcal{R}^o$, we denote the set of allowed
configurations (that is, configurations without FSC's) in the BTW model, which is also the
set of all recurrent configurations, see~\cite{meester}.
For $\eta$ and its corresponding integer part~$\overline{\eta}$, we have the following simple observation, the proof of which we leave to the reader.

\begin{lemma}
For $\eta\in\Omega$, $\eta$ is allowed if and only if $\overline{\eta}$ is  allowed in BTW.
 \label{relation2}
\end{lemma}

\begin{definition}
A configuration $\eta\in\Omega$ is called {\em reachable} if there exists a configuration $\eta'\in\mathcal{X}$ with $\eta'(x)\geq 1$ for all $x\in\Lambda$, such that
$$
\eta=\mathcal{S}\eta'.
$$
\end{definition}

\noindent  We denote the set of all
reachable configurations by~$\mathcal{R}'$.

\begin{theorem}
$\mathcal{R}=\mathcal{R}'$.
\label{fullconfiguration}
\end{theorem}

\begin{proof} Let $\eta\in\mathcal{R}$. By Lemma~\ref{relation2}, $\overline{\eta}$ is  BTW-allowed and therefore also recurrent 
(Theorem~5.4 in~\cite{meester}).  From \cite[Theorem 4.1]{meester}, we have that for every 
$x\in\Lambda$, there exists $n_x\geq 1$ such that
$$ 
a_x^{n_x}\overline{\eta}=\overline{\eta}.
$$
Therefore 
\begin{equation}
\prod_{x\in\Lambda}a_x^{2dn_x}\overline{\eta}=\mathcal{S}^o(\overline{\eta}+\sum_{x\in\Lambda}2dn_x\delta_x)=\overline{\eta}.
\label{So(xi)}
\end{equation}
Now let 
$$\xi=\overline{\eta}+\sum_{x\in\Lambda}2dn_x\delta_x$$
be a configuration in $\mathcal{X}^o$; note that $\xi(x)\geq 2d$ for all $x\in\Lambda$.
Let 
\begin{equation}
\eta'=\frac{1}{2d}\xi+\tilde{\eta}.
\end{equation}
Then $\eta' \in \mathcal{X}$ and 
\begin{enumerate}
\item $\eta'(x)\geq 1$ for all $x\in\Lambda$;
\item $\overline{\eta'}=\xi$;
\item $\widetilde{\eta'} = \tilde{\eta}$.
\end{enumerate}
We will now argue that $\mathcal{S}\eta'=\eta$.
>From Lemma~\ref{relation} and~\eqref{So(xi)} we have
$$
\overline{\mathcal{S}\eta'}=\mathcal{S}^o{\overline{\eta'}}=\mathcal{S}^o\xi=\overline{\eta}.
$$
Furthermore, since fractional parts are invariant under~$\mathcal{S}$, we have
$$
\widetilde{\mathcal{S}\eta'}=\widetilde{\eta'}=\tilde{\eta},
$$
and hence $\mathcal{S}\eta'=\eta$.

For the other direction, if $\eta\in\mathcal{R}'$, there exists an $\eta'\in\mathcal{X}$ with $\eta'(x)\geq 1$ for all $x\in\Lambda$,  and
such that
\begin{equation}
\eta=\mathcal{S}\eta'.
\end{equation}
Clearly, $\overline{\eta'}(x)\geq 2d$. By Lemma~\ref{relation}, we can now write
\begin{equation}
\overline{\mathcal{S}\eta'}=\mathcal{S}^o\overline{\eta'}=\prod_{x\in\Lambda}a_x^{\overline{\eta'}(x)-(2d-1)}\xi^{max},
\end{equation}
where $\xi^{max}(x)=2d-1$, for all $x\in\Lambda$.
This means that
$\overline{\mathcal{S}\eta'}$ is a configuration obtained by additions to~$\xi^{max}$. Since $\xi^{max}$ is allowed for the BTW model, by Theorem~5.4 of~\cite{meester}, $\xi^{max}$ is also recurrent.
But then $\overline{\mathcal{S}\eta'}$ is also recurrent and therefore allowed. It then follows that $\mathcal{S}\eta'$ is allowed for the CBTW model.
\end{proof}

\begin{corollary}
For $x\in\Lambda, u\in[0, 1)$, $\mathcal{R}$ is closed under the operation of~$A_x^u$.
\label{close}
\end{corollary}
\begin{proof}
For $\eta\in\mathcal{R}$, by Theorem~\ref{fullconfiguration}, there is a $\eta'\in\mathcal{X}$ with $\eta'(x)\geq 1$ for all $x\in\Lambda$, such that
$$
\eta=\mathcal{S}\eta'.
$$
By the abelian property of toppling operators,
$$
A_x^u\eta=\mathcal{S}(\mathcal{S}\eta'+u\delta_x)=\mathcal{S}(\eta'
+u\delta_x)
$$
and $\eta'+u\delta_x$ is  a configuration with only unstable sites.
It follows that $A_x^u\eta\in\mathcal{R}'$,
and hence $A_x^u\eta\in\mathcal{R}$.
\end{proof}

\begin{lemma}(1) $\Vol(\mathcal{R})=\det(\Delta)$;\\
(2) For every $\xi\in\mathcal{R}^o$,
$\mu(C(\xi))=\frac{1}{|\mathcal{R}^o|}$.
\end{lemma}

\begin{proof}
(1) We have
$$
\Vol(\mathcal{R})=\sum_{\xi\in\mathcal{R}^o}\Vol(C(\xi)).
$$
For each $\xi\in\mathcal{R}^o$,
$\Vol(C(\xi))=(2d)^{-|\Lambda|}$, hence
$\Vol(\mathcal{R})=|\mathcal{R}^o|(2d)^{-|\Lambda|}$.
In the BTW-sandpile model, it is well known that
$|\mathcal{R}^o|=\det(\Delta^o)$,  where $\Delta^o$ is the
toppling matrix, see \cite[Theorem~4.3]{meester}. Since $\Delta^o=2d\Delta$,  we have
$$
\Vol(\mathcal{R})=(2d)^{|\Lambda|}\det(\Delta)(2d)^{-|\Lambda|}=\det(\Delta).
$$
(2) is immediate from the definitions.
\end{proof}

\section{Proof of Theorem~\ref{invariant}}
 
We denote the sites in~$\Lambda$ by $x_1, x_2, \ldots,
x_{|\Lambda|}$ and define the collection~$\mathcal{L}$ as
$$
\mathcal{L} = \Bigl\{ [c, d)=\prod_{1\leq k\leq |\Lambda|}[c_k,d_k) : 0\leq c_k\leq d_k \leq \frac{1}{2d} \Bigr\},
$$
where $c=(c_1,\ldots, c_{|\Lambda|})$ and $d=(d_1,\ldots, d_{|\Lambda|})$. 
Note that $\mathcal{L}$ is a $\pi$-system.
For $\xi \in \mathcal{R}^o$, let
$$
C(\xi, [c, d))=\{\eta\in C(\xi): \tilde{\eta} \in [c, d)\},
$$
which is the set of configurations  whose integer part is
$\xi$ and whose fractional part is in the interval $[c, d)$.
Let 
\begin{equation}
\mathcal{I}=\{C(\xi, [c, d)): [c, d)\in \mathcal{L}, \xi\in\mathcal{R}^o \}.
\label{pi-system}
\end{equation}
$\mathcal{I}$ is also a $\pi$-system. In order to show that $\mu$ is $A_x^u$-invariant, it 
suffices to show that
\begin{equation}
\mu\{\eta: A_x^u\eta\in B\}=\mu(B)
\end{equation}
for all $ B\in\mathcal{I}$,
and we will do that by direct calculation.

Let $B=C(\xi, [c, d))$ and $\zeta\in B$, and define the configuration~$\eta$ 
by
\begin{equation}
	\tilde{\eta}(z) = \bigl( \zeta(z)-u\delta_x(z) \bigr) \bmod \frac1{2d}, 
	\qquad z\in\Lambda,
	\label{etatilde}
\end{equation}
and
\begin{equation}
	\overline{\eta} = \begin{cases}
		(a_x^{-1})^{\lfloor 2du \rfloor}\xi
		& \text{if $\tilde{\zeta}(x) \geq u\bmod \frac1{2d}$}; \\
		(a_x^{-1})^{\lfloor 2du\rfloor + 1}\xi
		& \text{if $\tilde{\zeta}(x) < u\bmod \frac1{2d}$}.
	\end{cases}
	\label{etabar}
\end{equation}
We claim that this $\eta$ is the unique $\eta \in \mathcal{R}$ such that 
$A_x^u\eta = \zeta$. To see this, first note that for any $\eta \in 
\mathcal{R}$, $\tilde{\eta}$ differs from $\widetilde{A_x^u\eta}$ only at the 
site~$x$, and that
\[ \widetilde{A_x^u\eta}(x) = (\eta(x)+u) \bmod \frac1{2d}. \]
This shows that for any $\eta\in\mathcal{R}$ such that $A_x^u\eta=\zeta$, its   
fractional part is given by~\eqref{etatilde}. Now there are two 
possibilities:
\begin{enumerate}
\item $(\widetilde{A_{x}^u\eta})(x) = \tilde{\zeta}(x) \geq 
	u\bmod\frac{1}{2d}$;
\item $(\widetilde{A_{x}^u\eta})(x) = \tilde{\zeta}(x) < u\bmod\frac{1}{2d}$.
\end{enumerate}
A little algebra reveals that in the first case,
\[
	\overline{A_{x}^u\eta} = 
	\mathcal{S}^o(\overline{\eta}+\lfloor2du\rfloor\delta_{x}) = 
	a_x^{\lfloor2du\rfloor}\overline{\eta},
\]
and in the second case,
\[
	\overline{A_{x}^u\eta} = 
	\mathcal{S}^o(\overline{\eta}+(\lfloor2du\rfloor+1)\delta_{x}) = 
	a_x^{\lfloor2du\rfloor+1}\overline{\eta}.
\]
In words, depending on the fractional part of~$\zeta$ at site~$x$, the 
addition of~$u$ to site~$x$ corresponds to either adding $\lfloor 2du \rfloor$ 
or $\lfloor 2du \rfloor +1$ particles at~$x$ in the BTW model. It follows that 
any $\eta\in\mathcal{R}$ such that $A_x^u\eta=\zeta$ must have integer part 
given by~\eqref{etabar}. We conclude that for every $\zeta\in B$, the $\eta$ 
defined by \eqref{etatilde} and~\eqref{etabar} is the unique $\eta \in 
\mathcal{R}$ such that $A_x^u\eta = \zeta$. It follows that $A_x^u$ is a bijection.

Next we show that $A_x^u$ preserves the measure $\mu$.
For this, note that the inverse image of~$B$ naturally partitions into two sets: if 
$x=x_i$, define $[c^1, d^1)$, $[c^2, d^2)$ to be intervals that differ from 
$[c, d)$ in the $i^{th}$~coordinate only, to the effect that $[c_{i}^1, 
d_{i}^1) = [c_{i}, d_{i}) \cap [u\bmod\frac{1}{2d}, \frac{1}{2d})$ and 
$[c_{i}^2, d_{i}^2) = [c_{i}, d_{i}) \cap [0, u\bmod\frac{1}{2d})$. Note that
$[c, d) = [c^1, d^1)\cup [c^2, d^2)$. From the above we have
\[
	\{\eta: A_x^u\eta\in C(\xi, [c^1, d^1))\}
	= C\bigl( (a_x^{-1})^{\lfloor2du\rfloor}\xi, [c', d') \bigr)
\]
and 
\[
	\{\eta: A_x^u\eta\in C(\xi, [c^2, d^2))\}
	= C\bigl( (a_x^{-1})^{\lfloor2du\rfloor+1}\xi, [c'', d'') \bigr),
\]
where $[c', d')$ and $[c'', d'')$ are intervals that differ from $[c,d)$ only 
in the $i^{th}$~coordinate to the effect that $[c'_{i}, d'_{i}) = [c^1_{i} - u 
\bmod \frac{1}{2d}, d^1_{i} - u \bmod \frac{1}{2d})$ and $[c''_{i}, d''_{i}) = 
[c_{i}^2 + \frac{1}{2d} - u \bmod \frac{1}{2d}, d_{i}^2 + \frac{1}{2d} - u 
\bmod \frac{1}{2d})$. Finally,
\[
	\Vol\bigl\{ C\bigl( (a_x^{-1})^{\lfloor2du\rfloor}\xi, [c', d') \bigr) 
	\bigr\}	= \Vol\bigl\{ C\bigl(\xi, [c^1, d^1)\bigr) \bigr\}
\]
and
\[
	\Vol\bigl\{ C\bigl( (a_x^{-1})^{\lfloor2du\rfloor+1}\xi, [c'', d'') \bigr) 
	\bigr\}	= \Vol\bigl\{ C\bigl(\xi, [c^2, d^2)\bigr) \bigr\},
\]
which implies that $\mu$ is $A_x^u$-invariant.
\qed

\section{Proof of Theorem~\ref{a<b}}

We prove Theorem~\ref{a<b} with a coupling.
Let $\eta, \zeta\in\Omega_\Lambda$ be two initial configurations and let $\eta_t,
\zeta_t$ be two copies of the CBTW sandpile starting from
$\eta, \zeta$ respectively (hence $\eta=\eta_0$ and $\zeta = \zeta_0$). 
Addition amounts at time~$t$
are $U_t^\eta, U_t^\zeta$ respectively , the addition sites are
$X_t^\eta, X_t^\zeta$ respectively. All these random quantities are independent of each other. In the proof, we 
will couple the two processes, and in the coupling, random variables will be written in `hat'-notation; see below.

For every $x\in\Lambda$ and $t=0,1,\ldots$, let $D_t(x)$ be defined as
\begin{equation}
D_t(x)=\frac{1}{\lceil\frac{4}{b-a}\rceil}\big(\eta_t(x)-\zeta_t(x)\big).
\end{equation}

We now define a coupling of the CBTW realisations starting from $\eta$ and~$\zeta$ respectively. The sites to which we add are
copied from the $\eta$-process, that is, we define
\begin{equation}
\hat{X}_t^\eta=\hat{X}_t^\zeta=X_t^\eta
\end{equation}
and write $X_t$ (without any superscript) for the common value.
The addition amount in the $\eta$-process is unchanged, that is, 
$$
\hat{U}_t^\eta=U_t^\eta,
$$
but in the $\zeta$-process we define
$$
\hat{U}_t^\zeta=[U_t^\eta+D_0(X_t)-a] \bmod (b-a)+a.
$$
Hence, in the coupling the $\eta$-process evolves as in the original version,  but the $\zeta$-process does not. It is not hard to see
that this definition gives the correct marginals and that if $U_t^{\eta}\in[\frac{3a+b}{4}, \frac{a+3b}{4}]$,
then
\begin{equation}
\hat{U}_t^\zeta=U_t^\eta+\frac{1}{\lceil\frac{4}{b-a}\rceil}(\eta_0(X_t)-\zeta_0(X_t)).
\end{equation}

Let us now say that event $\mathcal{O}$ occurs if
\begin{enumerate}
\item between times $1$ and $1 + |\Lambda| \lceil\frac{4}{b-a}\rceil$ (inclusive), all sites are chosen as addition sites exactly $\lceil\frac{4}{b-a}\rceil$ times;
\item between times $1$ and $1 + |\Lambda| \lceil\frac{4}{b-a}\rceil$ (inclusive), the addition amounts in the $\eta$-process are all
contained in the interval $[\frac{3a+b}{4}, \frac{a+3b}{4}]$.
\end{enumerate}

If $\mathcal{O}$ occurs, then we claim that at time $1 + |\Lambda| \lceil\frac{4}{b-a}\rceil$, the two processes are in the same state.
Indeed, by the abelian property we can obtain the configuration at time~$t$ by first adding all additions up to time~$t$, and after that topple
all unstable sites in any order. If $\mathcal{O}$ occurs, and we defer toppling to the very end, all heights will be the same in the two
processes by construction, and hence after toppling they remain the same. 

The probability that $\mathcal{O}$ occurs is uniformly bounded below, that is, uniformly in the initial configurations $\eta$ and~$\zeta$. Indeed,
all that is necessary is that all sites are equally often addition sites, and that the addition amounts are in the correct subinterval of $[a,b]$. These
two requirements are independent of the starting configurations. Of course, the evolution of the coupling does depend on the initial configurations via
the relations between the added amounts. That is where the subtlety of the present coupling lies. 

Hence, to finish our construction, we first see whether or not $\mathcal{O}$ occurs. If it does we are done. If it does not, we start all over
again, with the current configurations at time~$t$ as our new initial configurations, and $D_t(x)$ instead of $D_0(x)$. Continuing this way, we have a fixed positive probability
for success at each trial and therefore the two processes will almost surely be equal eventually. Due to the fact that the success probability
is uniformly bounded below, convergence will be exponentially fast. The result 
now follows by choosing $\eta$ and~$\zeta$ according to the distributions 
$\nu$ and~$\mu$, respectively.
\qed

\section{Proof of Theorems \ref{a=b} and~\ref{a=l/2d}}

We start with the proof of Theorem~\ref{a=b}.

\begin{proof}[Proof of Theorem~\ref{a=b}]
	Define the function~$g$ by
	\[
		g(\eta) = \exp\biggl(4d\pi i \sum_{x \in \Lambda} \eta(x)\biggr).
	\]
	It is easy to see that $g$ is continuous and, of course, bounded. Denoting 
	the distribution of the process at time~$t$ by~$\nu_t$, for $\nu_t$ to 
	converge weakly to~$\nu$, say, it must be the case that
	\begin{equation}
		\label{ff}
		\int g d\nu_t \to \int g d\nu.
	\end{equation}
	However, since at each iteration of the process, exactly one of the 
	fractional parts is increased by~$a$ and subsequently taken modulo $1/2d$, 
	we have that $\sum_{x \in \Lambda} \eta_t(x) \bmod 1/2d$ is $\nu_t$-surely 
	equal to $\bigl( \sum_{x \in \Lambda} \eta_0(x) + ta \bigr) \bmod 1/2d$, 
	and hence the sequence of integrals on the left of~\eqref{ff} does not 
	converge at all, unless $a$ is a multiple of $1/2d$.
\end{proof}

We now move on to the proof of Theorem~\ref{a=l/2d}. In order to describe the 
limiting measure~$\mu^\eta_a$ appearing in the statement of the theorem, we 
first introduce the following notation. For a measure~$\nu$ on the space of 
height configurations~$\mathcal{X}$, $\mathcal{S}\nu$ denotes the measure on 
the stable configurations~$\Omega$ defined by
\[ \mathcal{S}\nu(B) = \nu\{\eta\in\mathcal{X}: \mathcal{S}\eta\in B\}, \]
for every measurable set $B\subset\Omega$.

\begin{proof}[Proof of Theorem~\ref{a=l/2d}]
	When $a=0$, $\eta_t=\eta$ for all $t=1,2,\dotsc$, and the limiting 
	distribution $\mu^\eta_0$ is point mass at~$\eta$. 

	So assume $l>0$ and take $a=l/2d$. We first consider the case in which 
	$\eta(x)=0$ for all sites $x$, and proceed by a coupling between the BTW 
	and the CBTW model. To introduce this coupling, let $(X_t)$ denote a 
	common sequence of (random) addition sites, and define
	\[
		\theta_t   = \sum_{s=1}^t a\delta_{X_s}, \qquad
		\theta^o_t = \sum_{s=1}^t \delta_{X_s}.
	\]
	This couples the vector~$\theta_t$ of all additions until time~$t$ in the 
	CBTW model, with the vector~$\theta^o_t$ of all additions until time~$t$ 
	in the BTW model.

	In order to arrive at the BTW configuration at time~$t$, now we can first 
	apply all additions, and then topple unstable sites as long as there are 
	any. In doing so, we couple the topplings in the BTW model with those in 
	the CBTW model: if we topple a site~$x$ in the BTW model, we topple the 
	corresponding site in the CBTW model $l=2da$ times. If we topple at~$x$, 
	then in the BTW model, $x$ loses $2d$~particles, while all neighbours 	
	receive~1. In the CBTW model, $x$ loses total mass~$2da$ while all its 
	neighbours receive~$a$. Hence, the dynamics in the CBTW model is exactly 
	the same as in the BTW model, but multiplied by a factor of~$a$. In 
	particular, since all the topplings in the BTW model were legal, all the 
	topplings in the CBTW model must have been legal as well. Furthermore, 
	when the BTW model has reached the stable configuration $\kappa_t = 
	\mathcal{S}^o \theta^o_t$, then the corresponding configuration in the 
	CBTW model is simply $a\kappa_t$. However, $a\kappa_t$ need not be stable 
	in the CBTW model. Hence, in order to reach the CBTW configuration at 
	time~$t$, we have to stabilize $a\kappa_t$, leading to the stable 
	configuration $\rho_t = \mathcal{S} \theta_t$.

	From the results in~\cite{meester}, we know that the distribution 
	of~$\kappa_t$ converges exponentially fast (as $t \to \infty$) in total 
	variation to the uniform distribution on~$\mathcal{R}^o$. Hence, $\rho_t$ 
	is exponentially close to~$\mathcal{S} \nu_a$, where $\nu_a$ is the 
	uniform measure on the set $\{a\xi: \xi \in \mathcal{R}^o\}$.
	
	This settles the limiting measure in the case where we start with the 
	empty configuration. If we start with configuration~$\eta$ in the CBTW 
	model, we may (by abeliannes) first start with the empty configuration as 
	above---coupled to the BTW model in the same way---and then add the 
	`extra' $\eta$ to~$\rho_t$ at the end, and stabilize the obtained 
	configuration. It follows that the height distribution in the CBTW model 
	converges to $\mathcal{S} \nu_a^{\eta}$, where $\nu_a^{\eta}$ is the 
	uniform measure on the set $\{\eta + a\xi: \xi \in \mathcal{R}^o\}$.
\end{proof}

\section{Proof of Theorem~\ref{irrationala}}

The proof of Theorem~\ref{irrationala} is the most involved. It turns out that 
the viewpoint of random ergodic theory is very useful here, and we start by 
reformulating the sandpile in this framework; see~\cite{Ki} for a review of 
this subject.

We consider the CBTW with $a=b$, that is, with non-random additions at a 
randomly chosen site. As before, we denote by~$\mu$ the uniform measure 
on~$\mathcal{R}$. For every $x \in \Lambda$, we have a transformation $A_x^a: 
\mathcal{R} \to \mathcal{R}$ which we denote in this section by~$A_x$ (since 
$a$ is fixed). Recall that each~$A_x$ is a bijection by 
Theorem~\ref{invariant}. The system evolves by each time picking one of 
the~$A_x$ uniformly at random (among all $A_x$, $x \in \Lambda$) independently 
of each other. The product measure governing the choice of the subsequent 
transformations is denoted by~${\bf p}$, that is, ${\bf p}$ assigns 
probability $1/|\Lambda|$ to each transformation, and makes sure that 
transformations are chosen independently. The system has randomness
in two ways: an initial distribution~$\nu$ on~$\mathcal{R}$ and the choice of 
the transformations. To account for this we sometimes work with the product 
measure $\nu \times {\bf p}$. 

A probability measure~$\nu$ on~$\mathcal{R}$ is called {\em invariant} if
\[
	\nu(B)= \frac{1}{|\Lambda|} \sum_{x \in \Lambda} \nu(A_x^{-1}B),
\]
that is, $\nu$ preserves measure on average, not necessarily for each 
transformation individually. We call a bounded function~$g$
on~$\mathcal{R}$ $\nu$-\emph{invariant} if
\[
	\frac{1}{|\Lambda|} \sum_{x \in \Lambda} g \circ A_x(\eta) = g(\eta),
\]
for $\nu$-almost all $\eta \in \mathcal{R}$. We call a (measurable) subset~$B$ 
of~$\mathcal{R}$ invariant if its indicator function is a $\nu$-almost 
everywhere invariant function; this boils down to the requirement that
up to sets of $\nu$-measure~0, $B$ is invariant under each of the 
transformations~$A_x$ individually. Finally, we call an invariant probability 
measure~$\nu$ on~$\mathcal{R}$ {\em ergodic} if any $\nu$-invariant function 
is a $\nu$-a.s.\ constant. These definitions extend the usual definitions in 
ordinary (non-random) ergodic theory. It is well known and not hard to show 
(see \cite[Lemma~2.4]{Ki}) that $\nu$ is ergodic if and only if every 
invariant set has $\nu$-probability zero or one. As before, we denote by 
$\bar{\eta}$ and~$\tilde{\eta}$ the integer and fractional parts of the 
configuration~$\eta$.

\begin{lemma}
	\label{invariantsets}
	Let $\lambda$ be an invariant probability measure, $B$ a 
	$\lambda$-invariant set and $C$ a set such that $\lambda(C\diff B) = 0$. 
	Then $\lambda(A_x^{-1}C\diff B) = 0$ for all $x\in\Lambda$.
\end{lemma}

\begin{proof}
	From the assumptions and invariance of~$\lambda$, it follows that the sets 
	$A_x^{-1}B \diff B$, $A_x^{-1}(B\setminus C)$ and $A_x^{-1}(C\setminus B)$ 
	all have $\lambda$-measure~0. Now suppose $\omega \in A_x^{-1}C\setminus 
	B$. Then either $\omega \in A_x^{-1}B\setminus B$ or else $\omega \in 
	A_x^{-1}(C\setminus B)$. Since both these sets have $\lambda$-measure~0, 
	$\lambda(A_x^{-1}C\setminus B) = 0$. Next suppose that $\omega \in 
	B\setminus A_x^{-1}C$. Then either $\omega \in B\setminus A_x^{-1}B$ or 
	else $\omega \in A_x^{-1}(B\setminus C)$. Again, since both these sets 
	have $\lambda$-measure~0, $\lambda(B\setminus A_x^{-1}C) = 0$.
\end{proof}

A version of the following lemma is well known in ordinary ergodic theory,
see e.g.\ \cite[Proposition~5.4]{denker}.

\begin{lemma}
	\label{abscont}
	Let $\nu$ be an ergodic probability measure (in our sense) and let 
	$\lambda$ be an invariant probability measure which is absolutely 
	continuous with respect to~$\nu$. Then $\lambda$ is ergodic, and therefore 
	$\lambda=\nu$.
\end{lemma}

\begin{proof}
	For $u\in \mathbb{Z}_{\geq0}^\Lambda$, write
	\[	T_u := \prod_{x\in\Lambda} A_x^{u(x)}. \]
	By commutativity, the order of the composition is irrelevant. Now suppose 
	that $B$ is a $\lambda$-invariant set. Define
	\[
		C := \bigcap_{n=0}^\infty \bigcup_{u\in\mathbb{Z}_{\geq n}^\Lambda} 
		T_u^{-1}B.
	\]
	In words, $C$ is the set of configurations~$\eta$ such that for all~$n$, 
	there is a $u$ with $u(x)\geq n$ at every $x\in\Lambda$, for which 
	$T_u\eta \in B$. Since $T_u(A_x\eta) = T_{u+\delta_x}\eta$, it is not 
	difficult to see that $A_x\eta\in C$ if and only if $\eta\in C$. Hence 
	$A_x^{-1}C = C$ for all $x\in\Lambda$. In particular, $C$ is a 
	$\nu$-invariant set, so by ergodicity of~$\nu$, either $\nu(C)=0$ or 
	$\nu(C^c)=0$. But since $\lambda$ is absolutely continuous with respect 
	to~$\nu$, this implies that either $\lambda(C)=0$ or $\lambda(C^c)=0$.

	To prove that $\lambda$ is ergodic, it therefore suffices to show that 
	$\lambda(B) = \lambda(C)$, or equivalently, $\lambda(C\diff B)=0$. To this 
	end, we first claim that
	\begin{equation}
		(C\diff B) \subset \bigcup_{u\in\mathbb{Z}_{\geq0}^\Lambda} 
		(T_u^{-1}B\diff B).
		\label{difference}
	\end{equation}
	Indeed, if $\eta\in C\setminus B$, there must be a $u \in 
	\mathbb{Z}_{\geq0}^\Lambda$ such that $T_u\eta \in B$, hence $\eta \in 
	T_u^{-1}B \setminus B$. And if $\eta\in B\setminus C$, then there are only 
	finitely many~$n$ such that $T_u\eta\in B$ with $u(x)=n$ for all 
	$x\in\Lambda$. Hence $\eta\in B\setminus C$ implies that for some $n>0$ we 
	have $T_u\eta \not\in B$ with $u(x)=n$ for all $x\in\Lambda$, and 
	therefore $\eta\in B\setminus T_u^{-1}B$ for this particular~$u$. This 
	establishes~\eqref{difference}.

	To complete the proof, note that by repeated application of 
	Lemma~\ref{invariantsets}, it follows that for every $u \in 
	\mathbb{Z}_{\geq0}^\Lambda$, $\lambda(T_u^{-1}B \diff B) = 0$. Hence 
	$\lambda(C\diff B) = 0$ by~\eqref{difference}, and we conclude that 
	$\lambda$ is ergodic. Now it follows from Theorem~2.1 in~\cite{Ki} that 
	$\nu \times {\bf p}$ and $\lambda \times {\bf p}$ are also ergodic (in the 
	deterministic sense w.r.t.\ the skew product transformation). But since 
	$\lambda$ is absolutely continuous with respect to~$\nu$, $\lambda \times 
	{\bf p}$ is absolutely continuous with respect to~$\nu \times {\bf p}$, 
	and hence must be equal to $\nu \times {\bf p}$ by ordinary (non-random) 
	ergodic theory. It follows that $\nu = \lambda$.
\end{proof}

We will apply Lemma~\ref{abscont} with the uniform measure~$\mu$ in the role 
of~$\nu$. Before we can do so, we must first show that $\mu$ is ergodic, which 
is an interesting result in its own right.

\begin{theorem}
	\label{uniform}
	When $a=b \notin \mathbb{Q}$, the uniform measure~$\mu$ on~$\mathcal{R}$ 
	is ergodic.
\end{theorem}

\begin{proof}
	We denote the sites in~$\Lambda$ by $x_1, x_2, \dotsc, x_m$, and identify 
	a configuration $\eta \in [0,\infty)^\Lambda$ with the point $(\eta(x_1), 
	\dotsc, \eta(x_m))$ in $[0,\infty)^m$. By Proposition~3.1 
	in~\cite{meester}, there exists a $n = (n_1,\dotsc,n_m) \in 
	\mathbb{Z}_{\geq1}^m$ such that for all $i = 1,2,\dotsc,m$,
	\[
		a_{x_i}^{n_i} \xi = \xi, \qquad \forall\xi\in\mathcal{R}^o.
	\]
	Now let $\mathcal{A}$ be the rectangle
	\[
		\Bigl[ 0,\frac1{2d}n \Bigr)
		= \Bigl[ 0,\frac{n_1}{2d} \Bigr) \times \Bigl[ 0,\frac{n_2}{2d} \Bigr) 
		\times \dotsb \times \Bigl[ 0,\frac{n_m}{2d} \Bigr),
	\]
	and denote Lebesgue measure on~$\mathcal{A}$ by~$\lambda$. Write $\tau_i$ for 
	the translation by~$a$ modulo $n_i/2d$ in the $i^{th}$ coordinate 
	direction on~$\mathcal{A}$.

	A subset $D$ of~$\mathcal{A}$ is called $\lambda$-\emph{invariant} if 
	$\lambda(\tau_i^{-1}D) = \lambda(D)$ for all~$\tau_i$. We claim that for any 
	$\lambda$-invariant set~$D$, either $\lambda(D)=0$ or $\lambda(D^c)=0$. Although
    this fact is probably well known, we give a proof for completeness. 
    Write $I_D$ for the indicator function of~$D$, and for 
	$k\in\mathbb{Z}^m$, denote by~$c_k$ the Fourier coefficients of~$I_D$. 
	Then
	\[
		c_k = \int_{\mathcal{A}} I_D(\omega) f_k(\omega) \, d\lambda(\omega)
		= \int_D f_k(\omega) \, d\lambda(\omega),
	\]
	where
	\[
		f_k(\omega) = \prod_{j=1}^m \tfrac{2d}{n_j} e^{2\pi i \frac{2d}{n_j} 
		k_j \omega_j}, \qquad \omega\in\mathcal{A}.
	\]
	Now suppose that $k_j\neq 0$ for some $1\leq j\leq m$. Then, changing 
	coordinates by applying~$\tau_j$, we have that
	\[
		c_k
		= \int_{\tau_j^{-1}D} f_k(\tau_j\omega) \, d\lambda(\tau_j\omega)
		= \int_D f_k(\omega)e^{2\pi i \frac{2d}{n_j}k_j a}\, d\lambda(\omega)
		= c_k \, e^{4d\pi i \frac{k_j}{n_j} a},
	\]
	because Lebesgue measure is invariant under~$\tau_j$ and $D$ is 
	$\lambda$-invariant. But since $a$~is irrational, this implies that 
	$c_k=0$. It follows that $I_D = c_o$ $\lambda$-almost everywhere. Hence 
	$c_o$ is either 0 or~1, so that either $\lambda(D) = 0$ or $\lambda(D^c) = 
	0$.

	Next note that $\mathcal{A}$ is composed of the cubes
	\[
		C_k = \Bigl[ \frac{k_1}{2d},\frac{k_1+1}{2d} \Bigr) \times  \Bigl[ 
		\frac{k_2}{2d},\frac{k_2+1}{2d} \Bigr) \times \dotsb \times  \Bigl[ 
		\frac{k_m}{2d},\frac{k_m+1}{2d} \Bigr),
	\]
	where $k\in\mathbb{Z}^m$ with $0\leq k_i<n_i$ for $i=1,2,\dotsc,m$. Using 
	this fact, we define a map $\psi: \mathcal{A}\to \mathcal{R}$ as follows. 
	If $\omega$ is in the cube~$C_k$, let
	\[
		\xi_k = \prod\nolimits_{i=1}^m a_{x_i}^{k_i} \xi^{max},
	\]
	where $\xi^{max}(x) = 2d-1$ for all $x\in\Lambda$, and set
	\[
		\psi(\omega) := \omega - \frac1{2d}k + \frac1{2d}\xi_k.
	\]
	Thus, $\psi$ simply translates the cube~$C_k$ onto the cube of 
	configurations in~$\mathcal{R}$ whose integer part is~$\xi_k$. Notice that 
	multiple cubes in~$\mathcal{A}$ may be mapped by~$\psi$ onto the same cube 
	in~$\mathcal{R}$, but that $\psi: \mathcal{A} \to \mathcal{R}$ is 
	surjective, because every allowed configuration of the BTW model can be 
	reached from~$\xi^{max}$ after a finite number of additions (and 
	subsequent topplings). Furthermore, it is easy to see that a 
	translation~$\tau_i$ on~$\mathcal{A}$ corresponds to an addition~$A_{x_i}$ 
	on~$\mathcal{R}$, in the sense that
	\begin{equation}
		\label{psicommutes}
		\psi(\tau_i \omega) = A_{x_i} \psi(\omega).
	\end{equation}

	Now suppose that $B\subset\mathcal{R}$ is $\mu$-invariant. Since $\mu$ is 
	normalized Lebesgue measure on~$\mathcal{R}$ and $\lambda$ is Lebesgue 
	measure on~$\mathcal{A}$, we have that
	\[
		\lambda(\psi^{-1}(B)) = \sum_k \lambda(\psi^{-1}(B) \cap C_k)
		= \Vol(\mathcal{R}) \sum_k \mu(B \cap \psi(C_k)),
	\]
	where the sum is over all cubes in~$\mathcal{A}$. Because $\mu( 
	A_{x_i}^{-1}B \diff B) = 0$, this gives
	\[
		\lambda(\psi^{-1}(B))
		= \Vol(\mathcal{R}) \sum_k \mu(A_{x_i}^{-1}B \cap \psi(C_k))
		= \sum_k \lambda(\psi^{-1}(A_{x_i}^{-1}B) \cap C_k).
	\]
	By~\eqref{psicommutes}, we have that
	\[ \psi^{-1}(A_{x_i}^{-1}B)= \tau_i^{-1} \psi^{-1}(B) \]
	and it follows that $\lambda(\psi^{-1}(B)) = \lambda(\tau_i^{-1} 
	\psi^{-1}(B))$, hence $\psi^{-1}(B)$ is a $\lambda$-invariant set. 
	Therefore, either $\lambda(\psi^{-1}(B)) = 0$ or $\lambda( \psi^{-1}(B^c) 
	) = 0$. From the construction it then follows that either $\mu(B)=0$ or 
	$\mu(B^c)=0$. Therefore, $\mu$ is ergodic.
\end{proof}

The next lemma will be used to deal with the evolution of the joint 
distribution of the fractional parts. In order to state it we need a few 
definitions. Consider the unit cube $I_m= [0,1]^m$. Let $a$ be an irrational 
number and let $\theta$ be the measure which assigns mass~$1/m$ to each of the	
points $(a,0,\ldots,0), (0,a,0,\ldots, 0), \ldots, (0,\ldots, 0,a)$. Denote 
by~$\theta^{*n}$ the $n$-fold convolution of~$\theta$, where additions are 
modulo~1. Translation over the vector~$x$ (modulo~1 also) is denoted 
by~$\tau_x$. We define the measure~$\mu_N^x$ by
\[
	\mu_N^x = \frac{1}{N}\sum_{n=0}^{N-1} \theta^{*n}\tau^{-1}_x .
\]
In words, this measure corresponds to choosing $n \in \{0,1,\ldots, N-1\}$ 
uniformly, and then applying the convolution of $n$ independently chosen 
transformations with starting point~$x$. 

\begin{lemma}
	\label{jointfrac}
	For all $x$, 
	$\mu_N^x$ converges weakly to Lebesgue measure on~$I_m$ as $N\to\infty$.
\end{lemma}

\begin{proof}
	The setting is ideal for Fourier analysis. It suffices to prove (using 
	Stone-Weierstrass or otherwise) that
	\[ \int f_k\,d\mu_N^x \to \int f_k\,d\lambda, \]
	where $\lambda$ denotes Lebesgue measure and
	\[ f_k(y) = e^{2\pi i k\cdot y} = e^{2\pi i \sum_{j=1}^m k_jy_j}, \]
	for $k = (k_1,\dotsc,k_m) \in \mathbb{Z}^m$ and $y\in I_m$.

	When $k=(0,\dotsc,0)$, all integrals are equal to~1. For $k \neq 
	(0,\dotsc,0)$, $\int f_k\,d\lambda = 0$, and hence it suffices to prove 
	that for these~$k$,
	\[ \int f_k\,d\mu_N^x \to 0. \]
	 For $n\geq 0$ and $l = (l_1,\dotsc,l_m) \in \mathbb{Z}_{\geq0}^m$ such 
	 that $l_1+\dotsb+l_m = n$,
	\[
		\theta^{*n}\tau_x^{-1}\{x + la\}
		= \frac1{m^n} \binom{n}{l_1,l_2,\dotsc,l_m}.
	\]
	Hence, by the multinomial theorem,
	\[\begin{split}
		\int f_k\,d\theta^{*n}\tau_x^{-1}
		&= \sum_{l_1+\dotsb+l_m=n}\frac1{m^n} \binom{n}{l_1,l_2,\dotsc,l_m}
			e^{2\pi i k\cdot x} \prod_{j=1}^m e^{2\pi ia k_jl_j} \\
		&= f_k(x) \biggl( \sum_{j=1}^m \frac1{m} e^{2\pi iak_j} \biggr)^n
		 =: f_k(x) (\alpha_k)^n.
	\end{split}\]
	Note that $|\alpha_k|\leq1$ and $\alpha_k\neq1$ because $a$ is irrational. 
	Therefore,
	\[
		\int f_k\,d\mu_N^x = f_k(x) \frac1{N} \sum_{n=0}^{N-1} (\alpha_k)^n
		= f_k(x) \frac1{N} \frac{1-(\alpha_k)^N}{1-\alpha_k} \to 0,
	\]
	as required.
\end{proof}

We will use Lemma~\ref{jointfrac} to understand the fractional parts in the 
sandpile, working on the cube $[0,1/2d]^m$ for $m=|\Lambda|$ rather than on 
$[0,1]^m$. Since topplings have no effect on the fractional parts of the 
heights, for the fractional parts it suffices to study the additions, and a 
point $x \in [0,1/2d]^m$ corresponds to all fractional parts of the $m$~sites 
in the system.

\begin{proof}[Proof of Theorem \ref{irrationala}.]
We have from Theorem~\ref{uniform} that the uniform measure~$\mu$ 
on~$\mathcal{R}$ is ergodic. Suppose now that $\nu$ is another invariant 
measure. We will show that $\mu(B)=0$ implies $\nu(B)=0$, and according to 
Lemma~\ref{abscont} it then follows that $\nu=\mu$.

Consider the ``factor map''~$f$ defined via $f(\eta) = \tilde{\eta}$, that is, 
$f$ produces the fractional part when applied to a configuration. The map~$f$ 
commutes with the random transformations we apply to~$\eta$, in the sense 
that
\[ f(A_x(\eta)) =  \tilde{A}_x(f(\eta)), \]
where $\tilde{A}_x(\tilde{\eta})$ is the configuration of fractional parts 
that results upon adding~$a$ to the height at site~$x$. Write $\tilde{\nu} = 
\nu f^{-1}$ and $\tilde{\mu} = \mu f^{-1}$. Because $f$ commutes with the 
random transformations and $\nu$ is invariant, for any measurable 
subset~$\tilde{B}$ of $[0,1/2d]^{|\Lambda|}$ we have that
\[
	\nu(f^{-1}(\tilde{B}))
	= \frac1{|\Lambda|} \sum_{x\in\Lambda} \nu(A_x^{-1}f^{-1}(\tilde{B}))
	= \frac1{|\Lambda|} \sum_{x\in\Lambda} \nu(f^{-1}(\tilde{A}_x^{-1} 
	\tilde{B})).
\]
Therefore, we have 
\[
	\tilde{\nu}(\tilde{B}) = \frac1{|\Lambda|} \sum_{x\in\Lambda} 
	\tilde{\nu}(\tilde{A}_x^{-1} \tilde{B}),
\]
and hence $\tilde{\nu}$ is invariant.
We claim that $\tilde{\mu} = \tilde{\nu}$. To prove this, it suffices to show 
that $\tilde{\mu}(R) = \tilde{\nu}(R)$ for any rectangle~$R$. Since rectangles 
are $\tilde{\mu}$-continuity sets, by Lemma~\ref{jointfrac} and bounded 
convergence,
\[
	\tilde{\mu}(R) = \int \tilde{\mu}(R)\, d\tilde{\nu}
	= \int \lim_{N\to\infty} \int I_R \, d\mu_N^h \, d\tilde{\nu}(h)
	= \lim_{N\to\infty} \iint I_R \, d\mu_N^h \, d\tilde{\nu}(h).
\]
Here $\mu^h_N$ is the analogue of the measure defined in Lemma~\ref{jointfrac} 
on the space $[0,1/2d]^{|\Lambda|}$ rather than $[0,1]^m$; it corresponds to 
choosing $n \in \{0,1,\dotsc,N-1\}$ uniformly, and then applying~$n$ uniformly 
and independently chosen transformations to the fractional height 
configuration~$h$. Thus we also have that
\[\begin{split}
	\iint I_R \, d\mu_N^h \, d\tilde{\nu}(h)
	&= \int \frac1N \sum_{n=0}^{N-1} \frac1{|\Lambda|^n} 
	\sum_{y_1,\dotsc,y_n\in\Lambda} I_R(\tilde{A}_{y_1} \dotsm \tilde{A}_{y_n} 
	h) \, d\tilde{\nu}(h) \\
	&= \frac1N \sum_{n=0}^{N-1} \frac1{|\Lambda|^n} 
	\sum_{y_1,\dotsc,y_n\in\Lambda} \tilde{\nu}(\tilde{A}_{y_n}^{-1} \dotsm 
	\tilde{A}_{y_1}^{-1} R).
\end{split}\]
By invariance of~$\tilde{\nu}$, this last expression is equal to 
\[
    \frac1N \sum_{n=0}^{N-1} \tilde{\nu}(R)
	= \tilde{\nu}(R),
\]
and we conclude that $\tilde{\mu}(R) = \tilde{\nu}(R)$ for every rectangle~$R$, 
hence $\tilde{\mu}=\tilde{\nu}$.

For the final step, let $B$ be such that $\mu(B) = 0$ and write $\tilde{B} = 
f(B)$. We claim that then $\tilde{\mu}(\tilde{B}) = 0$. Indeed, the inverse 
image of~$\tilde{B}$ is a collection of points of the form $x+(2d)^{-1}\eta$, 
where $x \in B$ and $\eta \in \mathbb{Z}^\Lambda$. If $\mu(B)=0$, then also 
this collection of points has $\mu$-measure~0. Since $\tilde{\mu} = 
\tilde{\nu}$, it follows that $\tilde{\nu}(\tilde{B}) =0$, hence also 
$\nu(B)=0$. This concludes the proof.
\end{proof}

\paragraph{Acknowledgment.} It is a pleasure to thank Michael Keane for very 
interesting discussions.

\goodbreak

\end{document}